\documentclass[a4paper]{amsart}

\usepackage{amsmath,amssymb,amsthm,a4wide}

\usepackage{tikz}

\usetikzlibrary{calc,intersections,through,backgrounds}

\renewcommand{\Im}{\operatorname{Im}}
\renewcommand{\Re}{\operatorname{Re}}

\usepackage{hyperref,caption}
\usepackage{graphicx}
\usepackage{graphics}

\newtheorem{theorem}{Theorem}

\newtheorem{lemma}{Lemma}
\newtheorem*{corollary}{Corollary}

\newcommand{\diag}{\operatorname{diag}}

\begin{document}

\title[]{Refined Heinz-Kato-L\"owner inequalities}
\keywords{Cauchy-Schwarz inequality, L\"owner-Heinz inequality, Heinz-Kato inequality, Cordes inequality, McIntosh inequality.}
\subjclass[2010]{15A58 (primary), 47A30 and 47A63 (secondary)}

\author[]{Stefan Steinerberger}
\address[Stefan Steinerberger]{Department of Mathematics, Yale University, New Haven, CT 06510, USA}
\email{stefan.steinerberger@yale.edu}

\begin{abstract} A version of the Cauchy-Schwarz inequality in operator theory is the following: for any two symmetric, positive definite matrices $A,B \in \mathbb{R}^{n \times n}$ and
 arbitrary $X \in \mathbb{R}^{n \times n}$
$$ \|AXB\| \leq \|A^2 X\|^{\frac{1}{2}} \|X B^2\|^{\frac{1}{2}}.$$
This inequality is classical  and equivalent to the celebrated Heinz-L\"owner, Heinz-Kato and Cordes inequalities. We characterize cases of equality: in particular, after factoring out the symmetry coming from multiplication with
scalars $ \|A^2 X\| = 1 = \|X B^2\|$, the case of equality requires that $A$ and $B$ have a common eigenvalue $\lambda_i = \mu_j$. We also derive
improved estimates and show that if either $\lambda_i \lambda_j = \mu_k^2$ or $\lambda_i^2 = \mu_j \mu_k$ does not have a solution, i.e. if $d > 0$ where
\begin{align*}
d &= \min_{1 \leq i,j,k \leq n} \left\{  \left| \log{ \lambda_i} + \log{ \lambda_j} - 2\log{ \mu_k} \right|:\lambda_i, \lambda_j \in \sigma(A), \mu_k \in \sigma(B)    \right\} \\
&+\min_{1 \leq i,j,k \leq n}  \left\{  \left| 2\log{\lambda_i} - \log{ \mu_j} - \log{\mu_k } \right|:\lambda_i \in \sigma(A), \mu_j, \mu_k \in \sigma(B)    \right\},
\end{align*}
then there is an improved inequality
$$ \|AXB\| \leq (1 - c_{n,d})\|A^2 X\|^{\frac{1}{2}} \|X B^2\|^{\frac{1}{2}}$$
for some $c_{n,d} > 0$ that only depends only on $n$ and $d$. We obtain similar results for the McIntosh inequality and the Cordes inequality and 
expect the method to have many further applications.
\end{abstract}

\maketitle

\section{Introduction and statement of results}
\subsection{Introduction} We study the inequality
$$ \|AXB\| \leq \|A^2 X\|^{\frac{1}{2}} \|X B^2\|^{\frac{1}{2}},$$
where $A,B \in \mathbb{R}^{n \times n}$ are symmetric and positive definite, $X\in \mathbb{R}^{n \times n}$ is completely arbitrary and the norm is given by the operator norm
$$ \|A\| = \sup_{\|x\| =1}{\|Ax\|}.$$
Fujii \& Furuta \cite{fu} have shown that the special case
$$ \|PQP\| \leq \|P^2 Q\|^{\frac{1}{2}} \|Q P^2\|^{\frac{1}{2}}$$
($P,Q$ both symmetric and positive definite) is equivalent to the following inequalities:
\begin{itemize}
\item \textbf{L\"owner-Heinz inequality}. (L\"owner \cite{lowner}, 1934), (Heinz \cite{heinz}, 1951). If $A \geq B \geq 0$, then
$$ A^{\alpha} \geq B^{\alpha} \qquad \mbox{for all} \quad 0 \leq \alpha \leq 1$$
\item \textbf{Heinz-Kato inequality}. (Heinz \cite{heinz}, 1951), (Kato \cite{kato}, 1952). If $A,B$ are positive operators such that $\| Tx \| \leq \|Ax\|$ and
$\|T^{*} y\| \leq \|By\|$ for all $x,y \in H$, then
$$ \left| \left\langle Tx, y\right\rangle \right|  \leq \left\| A^{\alpha} x \right\| \left\|B^{1-\alpha} y\right\|   \qquad \mbox{for all}~0 \leq \alpha \leq 1$$
\item \textbf{Cordes inequality}. (Cordes \cite{cord}, 1987). For all symmetric and positive-definite $A,B$ and all $0 \leq s \leq 1$
$$ \left\| A^s B^s\right\| \leq \|AB\|^s.$$
\end{itemize}

Classical work inspired by the original paper of Heinz \cite{heinz} include a 1953 paper of Dixmier \cite{dix}, a 1955 paper of Heinz \cite{heinz2} and a 1960 paper of Cordes \cite{cord1}.
There are now hundreds of papers concerned with variations of these inequalities, for a first introduction we refer to Bhatia \& Kittaneh \cite{bha, bha2, bha3},  Furuta \cite{furuta1},
 the books of Furuta \cite{furuta2} and Zhan \cite{zhan} and references therein. 
A natural generalization of the inequality described above is due to A. McIntosh \cite{mc}. 
\vspace{5pt}
\begin{itemize}
\item \textbf{McIntosh inequality}. (McIntosh \cite{mc}, 1979). Let $A,B \in \mathbb{R}^{n \times n}$ be symmetric and positive definite, let $X\in \mathbb{R}^{n \times n}$ be arbitrary and
let $0 < r < 1$. Then 
$$ \|A^r XB^{1-r}\| \leq  \|A X\|^{r} \|X B\|^{1-r}.$$
\end{itemize}
\vspace{5pt}
This inequality is easily seen to contain the previous inequality as a special case (relabeling $A \rightarrow A^2, B \rightarrow B^2$ and $r = 1/2$) and, in particular, it implies all previously mentioned results.
Versions and variants in Hilbert space
have been given by Bhatia \& Davis \cite{bhad} and Kittaneh \cite{kit}. 

\subsection{Characterization of Equality.}
Despite a lot of activity surrounding these inequalities, the cases of equality are not known.
If we consider diagonal matrices, it is obvious that equality will only
result in conditions on some of the rows and cannot control the remaining entries (except that they have to be sufficiently small). More generally,
if equality occurs then it is generically stable under small perturbation of those subspaces which play no role in the computation of the first singular vector
and equality can therefore only impose restrictions on some of the subspaces.
It is instructive to consider the special case of diagonal matrices
$$ A = \diag(\lambda_1, \dots, \lambda_n), ~X=\diag(x_1, \dots, x_n) \quad \mbox{and} \quad B=\diag(\mu_1, \dots, \mu_n).$$ 
The natural renormalization $ \|A X\| = 1 = \|X B\|$ boils down to
$$ \max_{1\leq i \leq n}{ |\lambda_i x_i|} = 1 =  \max_{1\leq i \leq n}{ |x_i \mu_i|}$$
under which the inequality simplifies to
$$ \max_{1\leq i \leq n}{ \lambda_i^r |x_i| \mu_i^{1-r}} \leq 1.$$ 
The case of equality clearly requires the existence of a $1 \leq i \leq n$ for which $1/\lambda_i = |x_i| = 1/\mu_i$ (in which case a suitable vector $v$ would be given by having its only
nonzero entry at the $i-$th coordinate). We observe that this implies that $A$ and $B$ have to have a common eigenvalue and that if $\|AXBv\| = 1$, then $XBv$ is also an eigenvector of $A$. We show
this to be a general description of the case of equality.
As for notation, the spectral decomposition of $A$ and $B$ will be written as
$$ Av = \sum_{k=1}^{n}{ \lambda_k \left\langle a_k, v \right\rangle}a_k \qquad \mbox{and} \qquad Bv = \sum_{k=1}^{n}{ \mu_k \left\langle b_k, v \right\rangle}b_k.$$
Throughout the rest of the paper we use $\pi_{\mu}:\mathbb{R}^n \rightarrow \mathbb{R}^n$ to denote the spectral projector onto the space spanned by the eigenvalues of $B$ associated to the eigenvalue $\mu$.

\begin{theorem}[Equality in the McIntosh inequality.]  Let $A,X,B$ be as above and scaled in such a way that $ \|A X\| = 1 = \|X B\|$, let $v \in \mathbb{R}^n$ be normalized $\|v\|=1$ and
assume $$  \|A^r XB^{1-r} v\| = \|AX\|^{r} \|XB\|^{1-r} = 1.$$
If $X\pi_{\mu} v \neq 0$, then  $X\pi_{\mu}v$ is an eigenvector of $A$ with eigenvalue $\mu$. In particular, $A$ and $B$ have at least one common
eigenvalue.
\end{theorem}
This shows that the cases of equality are fairly restrictive: any nontrivial interaction between $A$ and $B$ that is created by $X$ has to have matching eigenvalues -- to a certain
extent this means that equality in the special case of diagonal matrices already paints a fairly complete picture of what can happen. As
expected, this characterization of equality does not impose any restriction on parts of spectra that play no role in determining the value of the operator norm. \\

The proof uses intuition coming from interpolation of analytic operators. A careful analysis of a suitably taylored holomorphic function shows that (under $ \|A X\| = 1 = \|X B\|$), cases of equality require
$$\forall ~t \in \mathbb{R} \qquad \|A^{1+it}XB^{-it} v\|^r= 1.$$
An instructive (albeit slightly inaccurate) visualization the underlying geometry is as follows: for any vector $v$ the matrix $B^{-it}$ introduces oscillation
separately in every eigenspace with a speed that is proportional to the eigenvalue. These rotations are then being mapped to $AXB^{-it}v$ which is then subjected to rotations induced
by $A^{it}$. The final result is a certain lack of cancellation which requires the rotations induced by $A^{it}$ to cancels those of $AXB^{-it}v$. This is only possible if their speed is
matched which requires common eigenvalues.

\subsection{Improved estimates.} Our approach also gives a way of deriving improved estimates: we show that as soon as one of the equations
$$ \lambda_i \lambda_j = \mu_k^2 \qquad \mbox{and} \qquad \lambda_i^2 = \mu_j \mu_k $$
does not have a solution, there is a quantitative improvement of the McIntosh inequality that only depends on how close the equation is to being solvable. Assuming, as we always do,
the normalization $ \|A X\| = 1 = \|X B\|$, we introduce a notion of distance between $\sigma(A)$ and $\sigma(B)$
\begin{align*}
d &= \min_{1 \leq i,j,k \leq n} \left\{  \left| \log{ \lambda_i} + \log{ \lambda_j} - 2\log{ \mu_k} \right|:\lambda_i, \lambda_j \in \sigma(A), \mu_k \in \sigma(B)    \right\} \\
&+\min_{1 \leq i,j,k \leq n}  \left\{  \left| 2\log{\lambda_i} - \log{ \mu_j} - \log{\mu_k } \right|:\lambda_i \in \sigma(A), \mu_j, \mu_k \in \sigma(B)    \right\}.
\end{align*}
Note that Theorem 1 implies that in the case of equality, there actually exists a solution to $\lambda_i^2 = \mu_j^2$ and we always have $d = 0$.
However, as soon as $d>0$, we can give a quantitative improvement that only depends on $n,r$ and $d$. 

\begin{theorem}[Refined McIntosh inequality] For every $0 < r < 1$ and every $d > 0$ there exists $c_{n,r,d} > 0$ such that
$$ \|A^r XB^{1-r}\| \leq (1-c_{n,r,d}) \|A X\|^{r} \|X B\|^{1-r}.$$
\end{theorem}
The constant $c_{n,r,d}$ arises the solution of a problem in approximation theory which we explain at the end of the paper. 
We are not aware of this particular problem ever having been treated before and prove only a basic result; improved results would imply better quantitative control on the size of $c_{n,r,d}$. The arguments in this paper  only imply
$$ c_{n,r,d} \geq c_{r} \exp{\left(-\frac{1}{c_{r}} \frac{\sqrt{n}}{d}\right)} \qquad \mbox{for some}~c_r>0~\mbox{depending on}~r.$$

The statement is easily illustrated: for any $a,b > 0$
\begin{align*} 
a^r &= \left\|  \begin{pmatrix} a & 0 \\ 0 & 0 \end{pmatrix}^{r}  \begin{pmatrix} 1 & 0 \\ 0 & 1 \end{pmatrix}  \begin{pmatrix} 1 & 0 \\ 0 & b \end{pmatrix}^{1-r} \right\| \\
&\leq   \left\|  \begin{pmatrix} a & 0 \\ 0 &0 \end{pmatrix}^{}  \begin{pmatrix} 1 & 0 \\ 0 & 1 \end{pmatrix} \right\|^{r} \left\|  \begin{pmatrix} 1 & 0 \\ 0 & 1 \end{pmatrix}   \begin{pmatrix} 1 & 0 \\ 0 & b \end{pmatrix}\right\|^{1-r}    = a^r  \max(1, b^{1-r})
\end{align*}
We see that the inequality is sharp for $0 \leq b \leq 1$ but that an improvement becomes possible as soon as $b > 1$ and the scale of the possible improvement depends only on $b$ and $r$.
We have
$$ \frac{\sigma(A)}{\|AX\|^{r}} = \left\{0, 1 \right\} \qquad \mbox \qquad  \frac{\sigma(B)}{\|XB\|^{1-r}} = \left\{\frac{1}{\max(1, b^{1-r})}, \frac{b}{\max(1,b^{1-r})}\right\}$$
and see that $d > 0$ if and only if $b > 1$ and $0 < r < 1$.

\subsection{Cordes inequality.} We emphasize that our approach is not limited to the McIntosh inequality. We illustrate this by treating the Cordes inequality \cite{cord} and obtain the same type of result using essentially the same argument: recall that the Cordes inequality states that
for all symmetric and positive-definite $A,B \in \mathbb{R}^{n \times n}$ and all $0 \leq s \leq 1$
$$ \left\| A^s B^s\right\| \leq \|AB\|^s.$$
This inequality has been of the continued interest (see, for example, applications related to the geometry of $C^*-$algebras \cite{andr, corach, corach1}), however, cases of equality are not known.
If $s \in \left\{0,1\right\}$, the inequality is trivially true and no information about $A$ and $B$ can be deduced. We write the spectral decomposition of $A$ and $B$ as
$$ Av = \sum_{k=1}^{n}{ \lambda_k \left\langle a_k, v \right\rangle}a_k \qquad \mbox{and} \qquad Bv = \sum_{k=1}^{n}{ \mu_k \left\langle b_k, v \right\rangle}b_k.$$
Under a suitable normalization $\|AB\| =1$, the case of equality requires $\lambda_i \mu_j = 1$ to have a solution. Moreover, the only way for the inequality to be attained for a vector $v$ is 'diagonal' 
action: all eigenvectors $b_k$ of $B$ for which $\left\langle v, b_k \right\rangle \neq 0$ are also eigenvectors of $A$ with the inverse eigenvalue, i.e. $A^s B^s b_k = A^s \mu_k^s b_k = \mu_k^s A^s b_k= b_k$.
 We also obtain an improvement as soon as either
$$ \frac{1}{\lambda_i} \frac{1}{\lambda_j} = \mu_k^2 \qquad \mbox{or} \qquad \frac{1}{\lambda_i^2} = \mu_j \mu_k \qquad \mbox{does not have a solution.}$$
More precisely, assuming again the rescaling $\| AB\| = 1$, we introduce a parameter $d^{*} \geq 0$ via
\begin{align*}
d^* &= \min_{1 \leq i,j,k \leq n} \left\{  \left| \log{ \lambda_i} + \log{ \lambda_j} + 2\log{ \mu_k} \right|:\lambda_i, \lambda_j \in \sigma(A), \mu_k \in \sigma(B)    \right\} \\
&+\min_{1 \leq i,j,k \leq n}  \left\{  \left| 2\log{\lambda_i} + \log{ \mu_j} + \log{\mu_k } \right|:\lambda_i \in \sigma(A), \mu_j, \mu_k \in \sigma(B)    \right\}.
\end{align*}

\begin{theorem}[Refined Cordes inequality] Let $0 < s < 1$, $ \|A B\| = 1 $, let $v \in \mathbb{R}^n$ be normalized $\|v\|=1$ and
assume 
$$  \|A^s B^{s} v\| = \|AB\|^s = 1.$$
If $\pi_{\mu}v \neq 0$, then $\pi_{\mu} v$ is an eigenvector of $A$ with eigenvalue $1/\mu$.
If $d^* > 0$, then
$$ \|A^s B^s\| \leq \left( 1 - c_{n,s,d^*}\right)  \|AB\|^s.$$
\end{theorem}
The precise behavior of $c_{n,s,d^*}$ is determined by the same problem in approximation theory that already determines $c_{n,s,d}$ in Theorem 2. The proof of Theorem 3 is merely a minor
variation of the previous arguments used in Theorem 1 and Theorem 2 -- indeed, we believe the argument to be applicable to a large class of inequalities.

\section{Proof of Theorem 1}

\subsection{A holomorphic function.} Given a symmetric, positive definite matrix $A:\mathbb{R}^n \rightarrow \mathbb{R}^n$ and a vector $v \in \mathbb{R}^n$, we write its spectral decomposition as
$$ Av = \sum_{k=1}^{n}{\lambda_k \left\langle v, a_k \right\rangle a_k}.$$
This has a natural extension to complex powers: for $z \in \mathbb{C}$, we define
$$ A^{z}v = \sum_{k=1}^{n}{\lambda_k^{z} \left\langle v, a_k \right\rangle a_k},$$
which, for a fixed vector $v$, is merely a vector whose entries are sums of complex exponentials.
Note that, for $\lambda \geq 0$,
$$ \lambda^{it} = \cos{(\left(\log{\lambda}\right)t)} + i  \sin{(\left(\log{\lambda}\right)t)} = e^{i (\log{\lambda}) t}.$$
In particular, we have that $A^{it}$ is a unitary matrix since
\begin{align*}
 \left\| \sum_{k=1}^{n}{\lambda_k^{i t} \left\langle v, a_k \right\rangle a_k}\right\|^2 &=  \left\|  \Re  \sum_{k=1}^{n}{\lambda_k^{i t} \left\langle v, a_k \right\rangle a_k} + \Im  \sum_{k=1}^{n}{\lambda_k^{i t} \left\langle v, a_k \right\rangle a_k}        \right\|^2 \\
&=   \left\|  \sum_{k=1}^{n}{  \cos{(\left(\log{\lambda_k}\right)t)}  \left\langle v, a_k \right\rangle a_k}\right\|^2 +  \left\| \sum_{k=1}^{n}{\sin{(\left(\log{\lambda_k}\right)t)}\left\langle v, a_k \right\rangle a_k}        \right\|^2  \\
&=     \sum_{k=1}^{n}{  \left[\cos^2{(\left(\log{\lambda_k}\right)t)} +\sin^2{(\left(\log{\lambda_k}\right)t)}\right] |a_k|^2}   =     \sum_{k=1}^{n}{  |a_k|^2} = \|v\|^2.
\end{align*}

 Since the desired inequality is invariant under multiplication by scalars, we can assume without loss of generality that
$$  \| A X \| = 1 = \|X B\|.$$
We will henceforth assume $v$ to be an arbitrary but fixed vector with $\|v\| = 1$.
The proof consists of a detailed analysis of the behavior of the vector-valued map
$$ z \rightarrow A^{1-z} X B^{z} v \qquad \mbox{on} \quad \left\{z \in \mathbb{C}: 0 \leq \Re z \leq 1\right\}.$$
We are mainly interested in the size of the arising vectors, which would motivate taking the complex inner product
$\left\langle w, w \right\rangle_{\mathbb{C}} := w \cdot \overline{w}$, however, this quantity is not holomorphic. Instead,
we consider, for fixed $\|v\|=1$, the complex-valued map
$$z \rightarrow  \left\langle A^{1-z} X B^{z} v,  A^{1-z} X B^{z} v  \right\rangle_{\mathbb{R}}  \quad \mbox{on} ~ \left\{z \in \mathbb{C}: 0 \leq \Re z \leq 1\right\},~\mbox{where}~ \left\langle v, w \right\rangle_{\mathbb{R}} = \sum_{k=1}^{n}{v_i w_i}.$$
We observe that every single entry in the vector
$$  A^{1-z} X B^{z} v \quad \mbox{is holomorphic and thus so is} \qquad   \left\langle A^{1-z} X B^{z} v,  A^{1-z} X B^{z} v\right\rangle_{\mathbb{R}}.$$
Of course, the Cauchy-Schwarz inequality is still valid since
$$ \left| \left\langle v, w \right\rangle_{\mathbb{R}} \right| =  \left|  \sum_{i=1}^{n}{v_i w_i} \right|  \leq  \sum_{i=1}^{n}{ \left| v_i w_i \right| } \leq   \left(\sum_{i=1}^{n}{ \left| v_i \right|^2 } \right)^{\frac{1}{2}}
 \left(\sum_{i=1}^{n}{ \left| w_i \right|^2 } \right)^{\frac{1}{2}} = \| v \|^{} \|w\|^{}.$$

We will now consider $z = 0 +i t$ for $t \in \mathbb{R}$. Using Cauchy-Schwarz, we see that
$$ \left| \left\langle A^{1-it} X B^{it} v,  A^{1-it} X B^{it} v\right\rangle \right| \leq \left\| A^{1-it} X B^{it} v \right\|^2.$$
$A^{-i t}$ and $B^{ it}$ are unitary matrices and $\left\| v \right\| = 1$, therefore
$$  \left\| A^{1-it} X B^{it} v \right\| =  \left\| A^{} X B^{it} v \right\| \leq \left\| A^{} X\right\| \left\| B^{it} v \right\| = 1.$$
The same reasoning implies that for $z = 1 + it$ 
$$  \left\| A^{it} X B^{1-it} v \right\| =   \left\| X B^{1-it} v \right\|   \leq  \left\|X B^{} \right\| \left\| B^{-it} v \right\|  = 1.$$
Using the trivial estimate  
\begin{align*}
 \|A^{1-z} X B^{z} v\| &\leq \|A^{1-z}\| \|X\| \|B^{z}\| \leq  \|A\|^{1-\Re z}   \|X\| \|B\|^{\Re z} \\
&\leq \max(\|A\|, 1) \|X\| \max( \|B\|, 1)
\end{align*}
it is easy to see that 
$$ \left| \left\langle A^{1-z} X B^{z} v,  A^{1-z} X B^{z} v  \right\rangle_{\mathbb{R}} \right| \leq 1 \qquad \mbox{ on} \quad \left\{z \in \mathbb{C}: 0 \leq \Re z \leq 1\right\}.$$
Let us now assume that indeed
$$ \|A^r XB^{1-r}\| =  \|A X\|^{r} \|X B\|^{1-r}= 1.$$
This means that for some vector $\|v\| = 1$, we have $ \| A^r XB^{1-r} v \| = 1$. Note that for this particular choice $z=1-r$ every single entry is real-valued and therefore
$$  \left\langle A^{r} X B^{1-r} v,  A^{r} X B^{1-r} v  \right\rangle_{\mathbb{R}}  = 1.$$
However, we are dealing with a holomorphic function that is uniformly bounded on the entire domain and has boundary values of size at most 1:  the maximum principle then implies that
$$ \left\langle A^{1-z} X B^{z} v,  A^{1-z} X B^{z} v  \right\rangle_{\mathbb{R}} =1  \qquad \mbox{for all} \quad \left\{z \in \mathbb{C}: 0 \leq \Re z \leq 1\right\}.$$
The next step is to analyze the implications of this equation for $ z = 0+ it$  (one could perform an equivalent analysis on $z = 1 + it$), where the equation simply becomes
$$ \left\langle A^{1 -  i t} X B^{ i t} v,  A^{1- i t} X B^{i t} v  \right\rangle_{\mathbb{R}} =1 \qquad \mbox{for all}~ t \in \mathbb{R}.$$
We now compute these matrices: 
$$ B^{it}v = \sum_{l = 1}^{n}{  \mu_l^{it}  \left\langle v, b_l \right\rangle b_l} \qquad \mbox{and thus} \qquad   XB^{it}v = \sum_{l = 1}^{n}{  \mu_l^{it}  \left\langle v, b_l \right\rangle X b_l}.$$
Furthermore
$$ A^{1 -  i t}  XB^{it}v = \sum_{k=1}^{n}{ \lambda_k^{1-it} \left\langle  XB^{it}v , a_k \right\rangle a_k}$$
and therefore
\begin{align*}
\left\langle A^{1 -  i t} X B^{ i t} v,  A^{1- i t} X B^{i t} v  \right\rangle_{\mathbb{R}}  &= \left\langle \sum_{k=1}^{n}{ \lambda_k^{1-it} \left\langle  XB^{it}v , a_k \right\rangle a_k},  \sum_{k=1}^{n}{ \lambda_k^{1-it} \left\langle  XB^{it}v , a_k \right\rangle a_k} \right\rangle_{\mathbb{R}} \\
&= \sum_{k=1}^{n}{ \lambda_k^{2-2it} \left\langle  XB^{it}v , a_k \right\rangle_{\mathbb{R}}^2}   \\
&= \sum_{k=1}^{n}{ \lambda_k^{2-2it} \left\langle  \sum_{l = 1}^{n}{  \mu_l^{it}  \left\langle v, b_l \right\rangle X b_l} , a_k \right\rangle_{\mathbb{R}}^2} \\
&= \sum_{k=1}^{n}{\lambda_k^{2-2it} \left(   \sum_{l=1}^{n} \mu_l^{it} \left\langle v, b_l \right\rangle   \left\langle X b_l, a_k \right\rangle  \right)^2 } = 1
\end{align*}

\subsection{Cases of Equality: Proof of Theorem 1}

\begin{proof}  The precise coefficients are not as important as the algebraic structure:
writing
$$ \alpha_k = \log{(\lambda_k)},~\beta_l = \log{(\mu_l)}, \quad \mbox{and} \quad c_{k,l} = \lambda_k \left\langle v, b_l \right\rangle   \left\langle X b_l, a_k \right\rangle$$
allows to notationally simplify the equation to
$$ \sum_{k=1}^{n}{  e^{- 2 \alpha_k i t} \left( \sum_{l=1}^{n}{ c_{k,l} e^{\beta_l i t } } \right)^2} =1.$$
Note that all coefficients are real-valued.
The remainder of the proof is algebraic: the only way for an expression of this type to be identically 1 for all values of $t \in \mathbb{R}$ is for all terms (except constants) to cancel.
We start with the assumption that every eigenvalue has multiplicity one and the eigenvalues are given by $\beta_1 < \dots < \beta_n$.\\

\textbf{Fact.} If $c_{k,l} \neq 0$, then $\alpha_k = \beta_l$. In particular, the equation $\alpha_k = \beta_l$ has a solution.\\

 Consider the quantities
$$
 \underline{\sigma} = \min \left\{ 2\beta_l - 2\alpha_k: c_{k,l} \neq 0 \right\} \qquad \mbox{and} \qquad
 \overline{\sigma} = \max \left\{ 2\beta_l - 2\alpha_k: c_{k,l} \neq 0 \right\}.
$$
These numbers give the smallest and largest occuring frequencies: an explicit expansion shows that the algebraic structure forces the leading coefficients to be sums of squares. This implies that there is no
form of cancellation and allows us to identify the smallest and largest occuring frequencies as $\underline{\sigma}$ and $\overline{\sigma}$:
$$
 \sum_{k=1}^{n}{ e^{- 2 \alpha_k i t} \left( \sum_{l=1}^{n}{ c_{k,l} e^{\beta_l i t } } \right)^2}  =  
 e^{i \underline{\sigma} t} \left(\sum_{k,l = 1 \atop 2\beta_l - 2\alpha_k = \underline{\sigma} }^{n}{c_{l,k}^2} \right) +  \sum_{j}{d_j e^{i e_j t}} + e^{i \overline{\sigma} t} \left(\sum_{k, l = 1 \atop 2\beta_l - 2\alpha_k = \overline{\sigma}}^{n}{c_{l,k}^2} \right)
$$
where the $d_j, e_j$ could be explicitly computed and the arising frequencies satisfy $\underline{\sigma} < e_j < \overline{ \sigma}$. However, in order for this expression to be 1, we require
$$ \underline{\sigma} = 0 = \overline{\sigma},$$
which was the desired statement. It remains to deal with the general case, where eigenvalues may have multiplicities.
Let us assume there are exactly $m$ distinct eigenvalues $\beta_1 < \dots < \beta_{m}$. Then we may write the equation as
 $$ \sum_{k=1}^{n}{  e^{- 2 \alpha_k i t} \left( \sum_{l=1}^{m}{   \left\langle X \pi_{\beta_{l}} v, a_k \right\rangle  e^{\beta_l i t } } \right)^2} =1$$
and the same argument as before applies: either $  \left\langle X \pi_{\beta_{l}} v, a_k \right\rangle = 0$ or $\alpha_k = \beta_{l}$. This means that if $X \pi_{\beta} v \neq 0$, then
it has to be mapped to the eigenspace of $A$ associated to the eigenvalue $\beta$.
\end{proof}

\section{Proof of Theorem 2}

\subsection{Two simple Lemmata.}
We need a fairly simple result stating that trigonometric functions oscillate to a certain extent on a fixed interval if their frequencies are bounded away from 0.

\begin{lemma}  For every $\delta > 0$ and $n \in \mathbb{N}$, any $d_1 < d_2 < \dots < d_n$ with $|d_j| \geq \eta$ and any $c_j \in \mathbb{R}$
$$ \int_{-2\sqrt{n}/\delta}^{2\sqrt{n}/\delta}{  \sum_{j=1}^{n}{c_j e^{i d_j t}} dt}  \leq  \frac{1}{2} \frac{4\sqrt{n}}{\delta} \left\| \sum_{j=1}^{n}{c_j e^{i d_j t}} \right\|_{L^{\infty}{(\mathbb{R})}}$$
\end{lemma}
\begin{proof}
It is easy to see that
\begin{align*} \left|  \frac{\delta}{4\sqrt{n}} \int_{-2\sqrt{n}/\delta}^{2\sqrt{n}/\delta}{  \sum_{j=1}^{n}{c_j e^{i d_j t}} dt} \right|
 &\leq \sum_{j=1}^{n}{|c_j| \left| \frac{\delta}{4\sqrt{n}} \int_{-2\sqrt{n}/\delta}^{2\sqrt{n}/\delta}{ e^{i d_j t}dt} \right| }  \\
&= \sum_{j=1}^{n}{   |c_j|     \sin{\left(\frac{2\sqrt{n} d_j}{\delta}\right)}   \frac{  \delta}{ 2\sqrt{n} d_j}} \\
&\leq \sup_{t \geq 2\sqrt{n}}{ \left| \frac{\sin{t}}{t} \right|}  \sum_{j=1}^{n}{ |c_j| } \leq \frac{1}{2\sqrt{n}}\sum_{j=1}^{n}{ |c_j| }  .
\end{align*}
It remains to bound the $L^{\infty}-$norm. We first observe that
$$ 2T \left\|\sum_{j=1}^{n}{c_j e^{i d_j t}}  \right\|^2_{L^{\infty}(\mathbb{R})} \geq  \int_{-T}^{T}{\left| \sum_{j=1}^{n}{c_j e^{i d_j t}}  \right|^2 dt} = \int_{-T}^{T}{\sum_{j=1}^{n}{c_j^2} + \sum_{j,k = 1 \atop j \neq k}^{n}{ c_j c_k e^{i(d_j-d_k)t}}dt}$$ 
and dividing by $T$ and letting $T \rightarrow \infty$ shows with Cauchy-Schwarz
$$ \frac{1}{2\sqrt{n}}\sum_{j=1}^{n}{ |c_j| } \leq \frac{1}{2} \left( \sum_{j=1}^{n}{c_j^2} \right)^{\frac{1}{2}} \leq \frac{1}{2} \left\| \sum_{j=1}^{n}{c_j e^{i d_j t}} \right\|_{L^{\infty}}.$$
\end{proof}
It could be desirable to prove a version of this inequality for shorter intervals of size $\sim \delta^{-1}$. Such a result would make our improvement $c_{n,r,d}$ independent of $n$ (although, of course, the
size of the matrix still enters implicitly through $d$).

\begin{lemma} Let $f:[0, \infty] \rightarrow \mathbb{R}^+$ be monotonically decreasing. If $g \in L^{\infty}(\mathbb{R})$ satisfies
$$ \int_{0}^{x}{g(z) dz} \leq  \frac{x}{2} \|g\|_{L^{\infty}(\mathbb{R})}, \quad \mbox{then} \qquad \int_{0}^{x}{f(z)g(z)dz} \leq \|g\|_{L^{\infty}(\mathbb{R})} \left(\int_{0}^{3x/4}{f(z)dz} - \int_{3x/4}^{x}{f(z)dz} \right).$$
\end{lemma}
\begin{proof} Before we sketch the argument, we find it helpful to display the extremal function $g$ that maximizes $\int_{0}^{x}{f(z)g(z)dx}$ (and does so for every nonnegative, monotonically decreasing $f$).
\begin{center}
\begin{figure}[h!]
\begin{tikzpicture}[scale=1]
\draw [->, thick] (0,0) -- (9,0);
\draw [->, thick] (0,-1) -- (0,2);
\draw [ thick] (0, 1.5) to[out = 0, in = 180] (8.5, 0.2);
\node at (2, 1.6) {$f$};
\draw [  thick] (0,1) -- (6,1);
\draw [  thick] (6,-1) -- (6,1);
\draw [ thick] (8,-1) -- (6,-1);
\node at (6.3, -0.8) {$g$};
\draw [ thick] (8,-0.1) -- (8,0.1);
\node at (8.1, -0.3) {$x$};
\node at (-0.9, 1) {$\|g\|_{L^{\infty}(\mathbb{R})}$};
\draw [ thick] (-0.1, 1) -- (0.1, 1);
\node at (-0.9, -1) {-$\|g\|_{L^{\infty}(\mathbb{R})}$};
\draw [ thick] (-0.1, -1) -- (0.1, -1);
\end{tikzpicture}
\caption{Maximizing the integral over $f \cdot g$ subject to constraints.}
\end{figure}
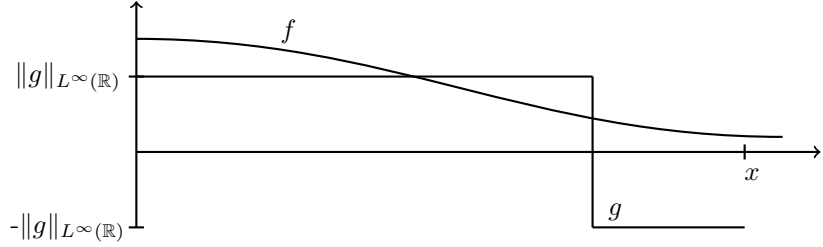
\end{center}
\vspace{-20pt}
The proof uses the classical Hardy-Littlewood rearrangement inequality \cite{hardy}, which implies that
$$  \int_{0}^{x}{f(z)g(z)dz} \leq  \int_{0}^{x}{f^*(z)g^*(z)dz},$$
where $f^*$ is the monotonically decreasing rearrangement of $f$ on $[0, x]$. Since $f$ is monotonically decreasing,
we have $f^* = f$. 
Consider now 
$$ h(z) = \begin{cases}  \|g\|_{L^{\infty}(\mathbb{R})} \qquad &\mbox{if}~0 \leq z \leq 3x/4 \\
-  \|g\|_{L^{\infty}(\mathbb{R})} \qquad &\mbox{if}~3x/4 \leq z \leq x.\end{cases}$$
It now suffices to show that
$$  \int_{0}^{x}{f(z)h(z)dz} -  \int_{0}^{x}{f(z)g^*(z)dz} =  \int_{0}^{x}{f(z)(h(z) - g^*(z))dz} \geq 0.$$
This follows immediately from the fact that
$$  \int_{0}^{x}{h(z) - g^*(z)dz} \geq 0,~~ \mbox{the definition of}~h~\mbox{and the monotonicity of}~f.$$
\end{proof}

We will use this Lemma on a symmetric interval $[-x,x] \subset \mathbb{R}$ with $f$ satisfying $f(x) = f(-x)$ in which case the result states that if
$$ \int_{-x}^{x}{g(z) dz} \leq  x \|g\|_{L^{\infty}(\mathbb{R})},~ \mbox{then} ~~~  \int_{-x}^{x}{f(z)g(z)dz} \leq  2\|g\|_{L^{\infty}(\mathbb{R})} \left(\int_{0}^{3x/4}{f(z)dz} - \int_{3x/4}^{x}{f(z)dz} \right).$$

\subsection{Proof of Theorem 2} We can now prove the refined inequality.

\begin{proof} Fix $0 < r < 1$ as well as the matrices $A, X$ and $B$. We can assume again
$$ \left\| AX \right\| = \left\|XB \right\| = 1$$
and $d>0$, where
\begin{align*}
 d &= \min_{1 \leq i,j,k \leq n} \left\{  \left| \log{ \lambda_i} + \log{ \lambda_j} - 2 \log{ \mu_k} \right|:\lambda_i, \lambda_j \in \sigma(A), \mu_k \in \sigma(B)    \right\} \\
&+\min_{1 \leq i,j,k \leq n}  \left\{  \left| 2\log{\lambda_i} - \log{ \mu_j} - \log{\mu_k } \right|:\lambda_i \in \sigma(A), \mu_j, \mu_k \in \sigma(B)    \right\}.
\end{align*}
It remains to show that for all normalized vector $\|v\| = 1$
$$ \| A^r X B^{1-r} v\| \leq 1 - c_{n,r,d}$$
for some constant $c_{n,r,d} > 0$ that only depends on these parameters. Consider the holomorphic map
$$ z \rightarrow  \left\langle A^{1-z} X B^{z} v,  A^{1-z} X B^{z} v  \right\rangle_{\mathbb{R}}  \qquad \mbox{on} \quad \left\{z \in \mathbb{C}: 0 \leq \Re z \leq 1\right\}.$$
It suffices to prove that this function at $z = 1- r + 0i$ is at most $1 - c_{n,r,d}$ for all $\|v\|=1$. Since the function is holomorphic, we can rewrite its value at a point as the
integral over the associated Poisson kernel $P$ paired with the boundary values of the function
$$  \left\langle A^{1-z} X B^{z} v,  A^{1-z} X B^{z} v  \right\rangle_{\mathbb{R}} = \int_{\partial  \left\{z \in \mathbb{C}: 0 \leq \Re z \leq 1\right\}}{ P(z) \left\langle A^{1-z} X B^{z} v,  A^{1-z} X B^{z} v  \right\rangle_{\mathbb{R}}} dz.$$
We will now restrict our attention to one of the two lines of the boundary depending on which of the two parts of the separation measure $d$ is bigger: we assume w.l.o.g. 
$$ \min \left\{  \left|2 \log{ \lambda_i} - \log{ \mu_j} - \log{ \mu_k} \right|:\lambda_i \in \sigma(A), \mu_j, \mu_k \in \sigma(B)    \right\} \geq \frac{d}{2}$$
and focus on $z = 0 + it$. If this was not the case, we would get the inequality for the other expression and focus on $z = 1 + it$. Since
$$  \left|  \left\langle A^{1-z} X B^{z} v,  A^{1-z} X B^{2z} v  \right\rangle_{\mathbb{R}}  \right| \leq 1 \qquad \mbox{on the entire strip},$$
it suffices to prove for some $c^{(2)}_{n,r,d} > 0$
$$ \int_{z = 0 + it}{  P(z) \left\langle A^{1-z} X B^{z} v,  A^{1-z} X B^{2z} v  \right\rangle_{\mathbb{R}} dz} \leq \left(1-c^{(2)}_{n,r,d}\right) \int_{z = 0 + it}{  P(z) dz}.$$
An expansion of the function shows that we can write it as
$$\left\langle A^{1-z} X B^{z} v,  A^{1-z} X B^{2z} v  \right\rangle_{\mathbb{R}} = \sum_{k=1}^{n}{\lambda_k^{2-2it} \left(   \sum_{l=1}^{n} \mu_l^{it} \left\langle v, b_l \right\rangle   \left\langle X b_l, a_k \right\rangle  \right)^2  } = \sum_{j}{c_j e^{i d_j t}},$$
where 
$$c_j \in \mathbb{R} \qquad \mbox{and} \qquad  d_j \in \left\{ \log{\mu_i} + \log{\mu_j} -  2\log{\lambda_k}: 1\leq i,j,k \leq n\right\} \subset \mathbb{R}$$  
and therefore $|d_j| \geq d/2$. We can invoke Lemma 1 and guarantee that on an interval $[-\ell, \ell] \subset i\mathbb{R}$ centered around the origin (where $\ell$ only depends on $n,r,d$)
$$  \int_{- i \ell}^{i \ell}{  \left\langle A^{1-z} X B^{z} v,  A^{1-z} X B^{2z} v  \right\rangle_{\mathbb{R}} dz} \leq \ell.$$
The Poisson kernel $P(z)$ is symmetric around the $x-$axis, has its local maximum at $z = 0$, is monotonically decaying away from that maximum and is nonnegative everywhere.
Lemma 2 applied to $f = P$ and $g$ being the trigonometric function the implies
 $$ \int_{- i \ell}^{i \ell}{  P(z) \left\langle A^{1-z} X B^{z} v,  A^{1-z} X B^{2z} v  \right\rangle_{\mathbb{R}} dz} \leq  \int_{- \frac{3i \ell}{4}}^{ \frac{3 i \ell}{4}}{P(z) dz} - 2   \int_{\frac{3i \ell}{4}}^{i \ell}{P(z) dz}$$
and this gives the desired result.
\end{proof}

\section{Proof of Theorem 3}
\begin{proof}  The argument is a straightforward adaption of the existing argument. We normalize $\|AB\| = 1$ and consider the complex-valued map
$$ z \rightarrow \left\langle A^{z} B^{z} v,  A^{z} B^{z} v \right\rangle \qquad \mbox{on} \quad \left\{z \in \mathbb{C}: 0 \leq \Re z \leq 1\right\}.$$
Trivially
$$  \left|   \left\langle  A^{it} B^{it} v,  A^{it} B^{it} v \right\rangle \right| \leq 1$$
because all matrices are unitary and we have
$$  \left|   \left\langle   A^{1+it} B^{1+it} v,  A^{1+it} B^{1+it} v \right\rangle \right| \leq 1$$
because of the normalization $\|AB\| = 1$. If indeed
$$ \|A^sB^sv\| = 1 \qquad \mbox{for some}~0<s<1,$$
then the maximum principle implies  that 
$$ \left\langle  A^{z} B^{z} v,  A^{z} B^{z} v \right\rangle=1 \qquad \mbox{on} \quad \left\{z \in \mathbb{C}: 0 \leq \Re z \leq 1\right\}.$$
We analyze this equation for $z = 0 + it$ for $t \in \mathbb{R}$. Clearly,
$$
\| A^{it}  B^{ it} v\|^2 = \sum_{k=1}^{n}{\lambda_k^{2it} \left(   \sum_{l=1}^{n} \mu_l^{it} \left\langle v, b_l \right\rangle   \left\langle  b_l, a_k \right\rangle  \right)^2 } =1.
$$
This falls within the general setup and the result follows from a repetition of previous arguments.
\end{proof}

\section{Some Remarks}
\subsection{Approximation Theory} This section states the approximation problem for which refined results would imply improved quantitative results about the behavior of
the constants $c_{n,r,d}$ and $c_{n,r,d^*}$ (and would also extend to other stability results obtained via the method outlined above).
Consider the strip
$ \left\{z \in \mathbb{C}: 0 \leq \Re z \leq \pi\right\},$
where the constant $\pi$ is chosen to simplify notation. A solution of the Dirichlet problem on that strip is given as follows (see e.g. \cite{widder}): define
$$ P(x,y) = \frac{\sin{x}}{\cosh{y} - \cos{x}}.$$
This function is nonnegative, decays exponentially and satisfies
$$ \int_{-\infty}^{\infty}{P(x,y) dy} = 2(\pi - x) \qquad \mbox{for all}~ 0<x<2\pi.$$
A solution of the Dirichlet problem on the strip is given by
$$ u(x,y) = \frac{1}{2\pi} \int_{-\infty}^{\infty}{ P(x,t-y) u(0, t) dt} +  \frac{1}{2\pi} \int_{-\infty}^{\infty}{ P(\pi-x,t-t-y) u(\pi, t) dt}.$$
Fix now $ 0 < r < 1$, define the function   $ g(y) = P(\pi r, y)/(2\pi(1 - r))$ (this normalization turns $g$ into a probability distribution) and the function class
$$ \mathcal{F} = \left\{  \sum_{k=1}^{n}{  e^{- 2 a_k i t} \left( \sum_{l=1}^{n}{ c_{k,l} e^{b_l i t } } \right)^2}:  a_k, b_l, c_{k,l} \in \mathbb{R} \right\},$$
which contains trigonometric functions with real coefficients that have a special algebraic structure. $\mathcal{F}$ contains the constant function 1 and therefore
$$ \sup_{f \in \mathcal{F} \atop \|f\|_{L^{\infty}(\mathbb{R})} \leq 1} \int_{-\infty}^{\infty}{f(y)g(y) dy} = 1.$$
The question is now under which conditions on $a_k$ and $b_k$ it is possible to obtain improved results.

Lemma 2 shows that if we define $\mathcal{H} \subset \mathcal{F}$ via
$$ \mathcal{H} = \left\{ f \in \mathcal{F}: \min_{1 \leq i,j,k \leq n}{ |b_i + b_j - 2a_k| \geq \delta} \right\},$$
then
$$ \sup_{f \in \mathcal{H} \atop \|f\|_{L^{\infty}(\mathbb{R})} \leq 1} \int_{-\infty}^{\infty}{f(y)g(y) dy} = 1 - c_{n,r,d} < 1 = \int_{-\infty}^{\infty}{g(y) dy} .$$
One natural question is whether the same result holds in the larger space  $\mathcal{H} \subset \mathcal{G} \subset \mathcal{F}$ given via
$$ \mathcal{G} = \left\{ f \in \mathcal{F}: \min_{1 \leq i,j \leq n}{ |a_i - b_j| \geq \delta} \right\},$$
which would then yield stability version that only depend on the minimal difference of elements in $\sigma(A)$ and $\sigma(B)$,
which would naturally complement our characterization of equality.
\begin{center}
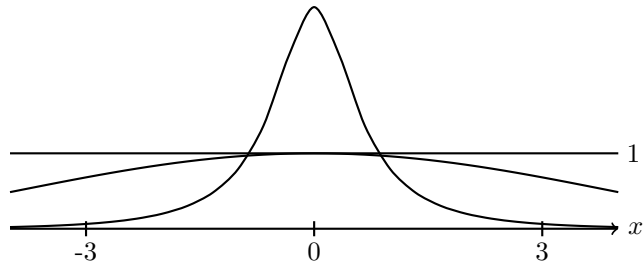
\begin{figure}[h!]
\begin{tikzpicture}
      \draw[->, thick] (-4,0) -- (4.,0) node[right] {$x$};
      \draw[thick, scale=1,domain=-4:4,smooth,variable=\x] plot ({\x},{ 0.588/(cosh(\x) - 0.8) });
\draw [thick] (0,-0.1) -- (0,0.1);
\node at (0, -0.3) {0};
\draw [thick] (3,-0.1) -- (3,0.1);
      \draw[scale=1,domain=-4:4, thick, smooth,variable=\x] plot ({\x},{ 1 });
\node at (4.2,1) {1};
\node at (3, -0.3) {3};
\draw [thick] (-3,-0.1) -- (-3,0.1);
\node at (-3, -0.3) {-3};
      \draw[thick, scale=1,domain=-4:4,smooth,variable=\x] plot ({\x},{ cos{deg(0.2*\x)}^2 });
    \end{tikzpicture}
\caption{An example for $r = 0.6$ and the function $\cos^2{(0.2x)} \in \mathcal{F}$.}
\end{figure}
\end{center}

\subsection{Simple Corollaries} A consequence of our equality characterization 
is that McIntosh inequalities with different indices are tightly linked. 

 \begin{corollary} Let $0<r<1$, $ \|A X\| = 1 = \|X B\|$, let $v \in \mathbb{R}^n$ be normalized $\|v\|=1$ and
assume $$  \|A^r XB^{1-r} v\| = 1.$$
Then for all $0 < s < 1$  $$  \|A^s XB^{1-s} v\| = 1.$$ 
\end{corollary}
 The crucial ingredient in our argument was that whenever the equation is attained, then
$$ \left\langle A^{1-z} X B^{z} v,  A^{1-z} X B^{z} v  \right\rangle_{\mathbb{R}} =1  \qquad \mbox{for all} \quad \left\{z \in \mathbb{C}: 0 \leq \Re z \leq 1\right\}.$$
This immediately implies the statement. The same argument also holds for the Cordes inequality.

 \begin{corollary} Let $0<r<1$, $ \|A B\| = 1$, let $v \in \mathbb{R}^n$ be normalized $\|v\|=1$ and
assume $$  \|A^r B^{r} v\| = 1.$$
Then for all $0 < s < 1$  $$  \|A^s B^{s} v\| = 1.$$ 
\end{corollary}

Another simple consequence of the proof is a slightly refined McIntosh inequality.

\begin{corollary} Let $A,B \in \mathbb{R}^{n \times n}$ be symmetric and positive definite and let $X  \in \mathbb{R}^{n \times n}$ be arbitrary. Then, for every $v \in \mathbb{R}^n$,
 $$\| A^rXB^{1-r}v\| \leq \sup_{t \in \mathbb{R}}{ \|A^{1+it}XB^{-it} v\|^r} \cdot \sup_{t \in \mathbb{R}}{ \|A^{it}XB^{1-it} v\|^{1-r}} \leq \|AX\|^r \|XB\|^{1-r}$$
\end{corollary}
This follows immediately from an application of the Hadamard three-line theorem.\\

\textbf{Acknowledgement.} This paper grew naturally out of a long series of enjoyable discussions with Raphy Coifman for which the author is very grateful.

\end{document}